\documentclass[12pt]{amsart}
\usepackage{amsmath, amsfonts, amssymb, amsthm,hyperref}
\hypersetup{hypertex=true,
	colorlinks=true,
	linkcolor=blue,
	anchorcolor=blue,
	citecolor=blue}
\usepackage{bm}
\allowdisplaybreaks[4]
\def\({\bg(}
\def\){\bg)}
\def\t{\text}

\def\ls{\leqslant}

\def\char{{\rm char}}
\def\1{{\bf 1}}

\def\pmod #1{\ ({\rm{mod}}\ #1)}

\def\Ack{\medskip\noindent {\bf Acknowledgments}}

\theoremstyle{plain}
\newtheorem{theorem}{Theorem}[section]
\newtheorem{lemma}{Lemma}
\newtheorem{corollary}{Corollary}
\newtheorem{conjecture}{Conjecture}

\theoremstyle{definition}

\theoremstyle{remark}

\makeatletter
\@namedef{subjclassname@2020}{%
	\textup{2020} Mathematics Subject Classification}
\makeatother
\vspace{4mm}

\begin{document}
	\medskip
	
	\title[On generalized Legendre matrices involving roots of unity over finite fields]
	{On generalized Legendre matrices involving roots of unity over finite fields}
	\author[N.-L. Wei, Y.-B. Li and  H.-L. Wu]{Ning-Liu Wei, Yu-Bo Li and Hai-Liang Wu}

	\address {(Ning-Liu Wei) School of Science, Nanjing University of Posts and Telecommunications, Nanjing 210023, People's Republic of China}
	\email{weiningliu6@163.com}
	
	\address {(Yu-Bo Li) School of Science, Nanjing University of Posts and Telecommunications, Nanjing 210023, People's Republic of China}
	\email{lybmath2022@163.com}
	
	\address {(Hai-Liang Wu) School of Science, Nanjing University of Posts and Telecommunications, Nanjing 210023, People's Republic of China}
	\email{\tt whl.math@smail.nju.edu.cn}
	
	\keywords{Legendre symbols, Finite Fields, Determinants.
		\newline \indent This work was supported by the Natural Science Foundation of China (Grant No. 12101321).}
	
	\begin{abstract}
		In this paper, motivated by the work of Chapman, Vsemirnov and Sun et al., we investigate some arithmetic properties of the generalized Legendre matrices over finite fields. For example, letting $a_1,\cdots,a_{(q-1)/2}$ be all non-zero squares in the finite field $\mathbb{F}_q$ which contains $q$ elements with $2\nmid q$, we give the explicit value of $D_{(q-1)/2}=\det[(a_i+a_j)^{(q-3)/2}]_{1\le i,j\le (q-1)/2}$. In particular, if $q=p$ is a prime greater than $3$, then 
		$$\left(\frac{\det D_{(p-1)/2}}{p}\right)=
		\begin{cases}
			1 & \mbox{if}\ p\equiv1\pmod4,\\
			(-1)^{(h(-p)+1)/2} & \mbox{if}\ p\equiv 3\pmod4\ \text{and}\ p>3,
		\end{cases}$$
		where $(\cdot/p)$ is the Legendre symbol and $h(-p)$ is the class number of $\mathbb{Q}(\sqrt{-p})$.
	\end{abstract}
	\maketitle
	
	\section{Introduction}
	\setcounter{lemma}{0}
	\setcounter{theorem}{0}
	\setcounter{equation}{0}
	\setcounter{conjecture}{0}
	\setcounter{remark}{0}
	\setcounter{corollary}{0}
	
	\subsection{Related Work and Motivations}
	Let $p$ be an odd prime and let $(\cdot/p)$ be the Legendre symbol. Chapman \cite{Chapman2004, Chapman2012} investigated the determinants involving Legendre matrices
	$$C_1=\left[\left(\frac{i+j-1}{p}\right)\right]_{1\le i,j\le (p-1)/2}$$ 
	and
	$$C_2=\left[\left(\frac{i+j-1}{p}\right)\right]_{1\le i,j\le (p+1)/2}.$$
	Surprisingly, these determinants are closely related to the quadratic fields. In fact, letting $\varepsilon_p>1$ and $h(p)$ be the fundamental unit and the class number of $\mathbb{Q}(\sqrt{p})$ and writing $\varepsilon_p=a_p+b_p\sqrt{p}$ with $a_b,b_p\in\mathbb{Q}$, Chapman \cite{Chapman2004} proved that 
	$$\det C_1=
	\begin{cases}
		(-1)^{(p-1)/4}2^{(p-1)/2}b_p & \mbox{if}\ p\equiv1\pmod4,\\
		0                            & \mbox{otherwise,}
	\end{cases}$$
	and 
	$$\det C_2=
	\begin{cases}
		(-1)^{(p+3)/4}2^{(p-1)/2}a_p   & \mbox{if}\ p\equiv1\pmod4,\\
		-2^{(p-1)/2}                   & \mbox{otherwise.}
	\end{cases}$$
	
	Later Chapman \cite{Chapman2012} posed the following conjecture.
	\begin{conjecture}[Chapman]
		Let $p$ be an odd prime and write 
		$\varepsilon_p^{(2-(2/p))h(p)}=a_p'+b_p'\sqrt{p}$ with $a_p',b_p'\in\mathbb{Q}$.
		Then 
		$$\det \left[\left(\frac{j-i}{p}\right)\right]_{1\le i,j\le (p+1)/2}
		=\begin{cases}
			-a_p' & \mbox{if}\ p\equiv1\pmod 4,\\
			1     & \mbox{otherwise.}
		\end{cases}$$
	\end{conjecture}
	
	Due to the difficulty of this conjecture, Chapman called this determinant ``the evil determinant". In 2012 and 2013, Vsemirnov \cite{Vsemirnov2012, Vsemirnov2013} confirmed this conjecture (the case $p\equiv3\pmod4$ in \cite{Vsemirnov2012} and the case $p\equiv 1\pmod4$ in \cite{Vsemirnov2013}). 
	
	Along this line, in 2019 Sun \cite{S19} studied some variants of Chapman's determinants. For example, let 
	$$S(d,p)=\det\left[\left(\frac{i^2+dj^2}{p}\right)\right]_{1\le i,j\le (p-1)/2}.$$
	Sun \cite[Thm. 1.2]{S19} showed that $S(d,p)=0$ whenever $(d/p)=-1$ and that $(-S(d,p)/p)=1$ whenever $(d/p)=1$. Readers may refer to \cite{krachun, Li-Wei,W21, Wu-She-Wang} for the recent progress on this topic. Also, Sun \cite[Thm. 1.4]{S19} proved that 
	\begin{equation}\label{Eq. Sun determinant A}
		\det\left[\frac{((i+j)/p)}{i+j}\right]_{1\le i,j\le (p-1)/2}
		\equiv\begin{cases}
			(2/p)\pmod p       & \mbox{if}\ p\equiv1\pmod4,\\
			((p-1)/2)!\pmod p  & \mbox{otherwise,}
		\end{cases}
	\end{equation}
	and that 
	$$\det\left[\frac{1}{i^2+j^2}\right]_{1\le i,j\le (p-1)/2}\equiv (-1)^{\frac{p+1}{4}}\pmod p$$
	whenever $p\equiv3\pmod4$. In 2022, the third author and Wang \cite[Thm. 1.7]{Wu-Wang} considered the determinant $\det[1/(\alpha_i+\alpha_j)]_{1\le i,j\le (p-1)/k}$, where 
	$0<\alpha_1,\cdots,\alpha_{(p-1)/k}<p$ are all $k$th power residues modulo $p$ and showed that for any positive even integer $k\mid p-1$, if $-1$ is not a $k$th power modulo $p$, then  
	$$\det\left[\frac{1}{\alpha_i+\alpha_j}\right]_{1\le i,j\le m}
	\equiv \frac{(-1)^{(m+1)/2}}{(2k)^m}\pmod p,$$
	where $m=(p-1)/k$. 
	
	Now let $\mathbb{F}_q$ be the finite field of $q$ elements with $\char(\mathbb{F}_q)=p>2$. It is known that $\mathbb{F}_q^{\times}=\mathbb{F}_q\setminus\{0\}$ is a cyclic group of order $q-1$ and that the subgroups
	$$U_k=\{x\in\mathbb{F}_q: x^k=1\}=\{a_1,\cdots,a_k\}\ (k\ge 1, k\mid q-1)$$
	are exactly all subgroups of $\mathbb{F}_q^{\times}$. Let $\phi$ be the unique quadratic character of $\mathbb{F}_q$, i.e., 
	$$\phi(x)=\begin{cases}
		 1 & \mbox{if}\ x\ \text{is a non-zero square},\\
		 0 & \mbox{if}\ x=0,\\
		-1 & \mbox{otherwise.}
	\end{cases}$$
    As $\char(\mathbb{F}_q)>2$, the subset $\{\pm1\}\subseteq\mathbb{Z}$ can be viewed as a subset of $\mathbb{F}_q$. From now on, we always assume $\pm1\in\mathbb{F}_q$. Inspired by Sun's determinant (\ref{Eq. Sun determinant A}), it is natural to consider the matrix $[\frac{\phi(a_i+a_j)}{a_i+a_j}]_{1\le i,j\le k}$. However, if $k\mid q-1$ is even, then the denominator $a_i+a_j=0$ for some $i,j$ since $-1\in U_k$ in this case. To overcome this obstacle, note that for any $x\in\mathbb{F}_q$ we have $\phi(x)=x^{(q-1)/2}$. Hence in this paper, we first focus on the matrix 
    $$D_k:=\left[\left(a_i+a_j\right)^{(q-3)/2}\right]_{1\le i,j\le k}.$$
	The main results involving $D_k$ will be given in Section 1.2. 
	
	We now consider another type of determinants. Sun \cite[Remark 1.3]{S19} posed the following conjecture.
	\begin{conjecture}[Sun]
		Let $p\equiv2\pmod3$ be an odd prime. Then 
		\begin{equation}\label{Eq. Sun determinant B}
			2\det\left[\frac{1}{i^2-ij+j^2}\right]_{1\le i,j\le p-1}
		\end{equation}
		is a quadratic residue modulo $p$. 
	\end{conjecture}

	Later the third author, She and Ni \cite{WSN} obtained the following generalized result.
	\begin{theorem}[Wu, She and Ni]
		Let $q\equiv2\pmod3$ be an odd prime power. Let $\beta_1,\cdots,\beta_{q-1}$ be all non-zero elements of $\mathbb{F}_q$. Then 
		$$\det\left[\frac{1}{\beta_i^2-\beta_i\beta_j+\beta_j^2}\right]_{1\le i,j\le q-1}
		=(-1)^{(q+1)/2}2^{(q-2)/3}\in\mathbb{F}_p,$$
		where $p=\char(\mathbb{F}_q)$.
	\end{theorem}
	
	Recently, Luo and Sun \cite{Luo-Sun} investigated the determinant
	\begin{equation}\label{Eq. Sun determinant C}
		\det S_p(c,d)=\det\left[\left(i^2+cij+dj^2\right)^{p-2}\right]_{1\le i,j\le p-1}.
	\end{equation}
	For $(c,d)=(1,1)$ or $(2,2)$, they determined the explicit values of $(\frac{\det S_p(c,d)}{p})$. 
	
    Motivated by Sun's determinants (\ref{Eq. Sun determinant A})--(\ref{Eq. Sun determinant C}) and the above discussions, in this paper, we also consider the matrix 
	$$T_k:=\left[\left(a_i^2+a_ia_j+a_j^2\right)^{(q-3)/2}\right]_{1\le i,j\le k}.$$
	
	We will state our results concerning $T_k$ in Section 1.3. 
	
	\subsection{The Main Results involving $\det D_k$}
	\begin{theorem}\label{Thm. the determinant of Dk}
		Let $\mathbb{F}_q$ be the finite field of $q$ elements with $\char(\mathbb{F}_q)=p>2$. Then for any integer $k\mid q-1$ with $1<k\le q-1$ we have 
		$$\det D_k=(-1)^{(k+1)(q-3)/2}\cdot w_k\cdot k^k\in\mathbb{F}_p,$$
		where 
		$$w_k=\prod_{s=0}^{k-1}\sum_{r=0}^{\lfloor\frac{q-3-2s}{2k}\rfloor}	\binom{(q-3)/2}{s+rk}\in\mathbb{F}_p.$$
	\end{theorem}
	
	Suppose now that $k=(q-1)/2$, i.e., $U_{(q-1)/2}$ is the set of all non-zero squares over $\mathbb{F}_q$. Then we can obtain the following simplified result which will be proved in Section 2. 
	
	\begin{corollary}\label{Cor. B of Thm.1.1}
		Let $\mathbb{F}_q$ be the finite field of $q$ elements with $\char(\mathbb{F}_q)=p>2$. Then 
		$$\det D_{(q-1)/2}=
		\begin{cases}
	(-1)^{\frac{q+3}{4}}u^2                         & \mbox{if}\ q\equiv 1\pmod4,\\
	(-1)^{\frac{q+5}{4}}\binom{(q-3)/2}{(q-3)/4}v^2 & \mbox{if}\ q\equiv 3\pmod4\ \text{and}\ q>3,
		\end{cases}$$
	where $u,v\in\mathbb{F}_p$ are defined by 
	$$u=\prod_{s=0}^{(q-5)/4}\binom{(q-3)/2}{s},\ \text{and}\ v=\prod_{s=0}^{(q-7)/4}\binom{(q-3)/2}{s}.$$
	In particular, if $q=p>3$ is an odd prime, then $D_{(p-1)/2}$ is non-singular and 
	$$\left(\frac{\det D_{(p-1)/2}}{p}\right)=
	\begin{cases}
		1 & \mbox{if}\ p\equiv1\pmod4,\\
		(-1)^{(h(-p)+1)/2} & \mbox{if}\ p\equiv 3\pmod4\ \text{and}\ p>3,
	\end{cases}$$
where $h(-p)$ is the class number of $\mathbb{Q}(\sqrt{-p})$. 
	\end{corollary}
	
	From Theorem \ref{Thm. the determinant of Dk} we see that $\det D_k\in\mathbb{F}_p$. The next result gives the explicit value of $(\frac{\det D_k}{p})$ when $k$ is odd. 
	
	\begin{theorem}\label{Thm. the legendre symbol of Dk, k odd}
		Let $\mathbb{F}_q$ be the finite field of $q$ elements with $\char(\mathbb{F}_q)=p>2$. Let $1<k\le q-1$ be an odd integer with $k\mid q-1$. Suppose that $D_k$ is non-singular. Then 
		$$\left(\frac{\det D_k}{p}\right)=\left(\frac{s_k}{p}\right),$$
		where 
		$$s_k:=k\sum_{r=1}^{\frac{q-1}{2k}}\binom{(q-3)/2}{((2r-1)k-1)/2}\in\mathbb{F}_p.$$
	\end{theorem}
	
	\subsection{The Main Results involving $\det T_k$}
	To state next results, we need to introduce some basic properties of trinomial coefficients. Let $n$ be a positive integer. For any integer $r$, the trinomial coefficient $\binom{n}{r}_2$ is defined by 
	$$\left(x+\frac{1}{x}+1\right)^{n}=\sum_{r=-\infty}^{+\infty}\binom{n}{r}_2x^r.$$
	This implies that $\binom{n}{r}_2=0$ whenever $|r|>n$ and that $\binom{n}{r}_2=\binom{n}{-r}_2$ for any integer $r$. In particular, $\binom{n}{0}_2$ is usually called the central trinomial coefficient because $\binom{n}{0}_2$ is exactly the coefficient of $x^n$ in the polynomial $(x^2+x+1)^n$. For simplicity, $\binom{n}{0}_2$ is also denoted by $t_n$. 
	
		\begin{theorem}\label{Thm. the determinant of Tk}
		Let $\mathbb{F}_q$ be the finite field of $q$ elements with $\char(\mathbb{F}_q)=p>2$. Then for any integer $k\mid q-1$ with $1<k\le q-1$ we have 
		$$\det T_k=l_k\cdot k^k\in\mathbb{F}_p,$$
		where 
		$$l_k=\prod_{s=0}^{k-1}\sum_{r=0}^{\lfloor\frac{q-3-s}{k}\rfloor}	\binom{(q-3)/2}{(q-3)/2-s-kr}_2\in\mathbb{F}_p.$$
	\end{theorem}
	
	As a direct consequence of Theorem \ref{Thm. the determinant of Tk} we have the following result.
	\begin{corollary}\label{Cor. A of Thm.1.3}
		Let $\mathbb{F}_q$ be the finite field of $q$ elements with $\char(\mathbb{F}_q)=p>2$. For any integer $k\mid q-1$ with $1<k\le q-1$, the matrix $T_k$ is singular over $\mathbb{F}_q$ if and only if 
		$$ \sum_{r=0}^{\lfloor\frac{q-3-s}{k}\rfloor}	\binom{(q-3)/2}{(q-3)/2-s-kr}_2\equiv0\pmod p$$
		for some $0\le s\le k-1$. In particular, $T_{q-1}$ is a singular matrix over $\mathbb{F}_q$.
	\end{corollary}
	
	In the case $k=(q-1)/2$, similar to Corollary \ref{Cor. B of Thm.1.1}, by Theorem \ref{Thm. the determinant of Tk} we directly have the following simplified result. 
	
	\begin{corollary}\label{Cor. C of Thm.1.3}
		Let $\mathbb{F}_q$ be the finite field of $q$ elements with $\char(\mathbb{F}_q)=p>2$. 
		
		{\rm (i)} If $q\equiv1\pmod4$, then 
		$$\det T_{(q-1)/2}=
		\prod_{s=0}^{(q-5)/4}
		\left(\binom{(q-3)/2}{(q-3)/2-s}_2+\binom{(q-3)/2}{1+s}_2\right)^2.$$
		
		{\rm (ii)} If $q\equiv3\pmod4$ and $q>3$, then 
		$$\det T_{(q-1)/2}=
		\binom{(q-3)/2}{0}_2\prod_{s=0}^{(q-7)/4}
		\left(\binom{(q-3)/2}{(q-3)/2-s}_2+\binom{(q-3)/2}{1+s}_2\right)^2.$$
		
		In particular, if $T_{(q-1)/2}$ is non-singular, then 
		$$
		\left(\frac{\det T_{(q-1)/2}}{p}\right)=\begin{cases}
			(-1)^{(q-1)/4}                                   & \mbox{if}\ q\equiv1\pmod4,\\
			(\frac{t_{(q-3)/2}}{p})(-1)^{(q+5)/4}            & \mbox{if}\ q\equiv3\pmod4\ \text{and}\ q>3.
		\end{cases}
		$$
	\end{corollary}
	
	\section{Proofs of Theorems \ref{Thm. the determinant of Dk} and Corollary \ref{Cor. B of Thm.1.1}}
	
	We begin with the following result (see \cite[Lemma 10]{K2}).
	\begin{lemma}\label{Lem. K}
		Let $R$ be a commutative ring. Let $P(t)=p_0+p_1t+\cdots+p_{n-1}t^{n-1}\in R[t]$. Then 
		$$\det[P(X_iY_j)]_{1\le i,j\le n}=\prod_{i=0}^{n-1}p_i\times
		\prod_{1\le i<j\le n}\left(X_j-X_i\right)\left(Y_j-Y_i\right).$$
	\end{lemma}
	
	We also need the following result.
	\begin{lemma}\label{Lem. product involving ai}
			Let $\mathbb{F}_q$ be the finite field of $q$ elements with $\char(\mathbb{F}_q)=p$. For any positive integer $k\mid q-1$, if we set $U_k=\{a_1,\cdots,a_k\}$, then 
			$$\prod_{1\le i<j\le k}\left(a_j-a_i\right)\left(\frac{1}{a_j}-\frac{1}{a_i}\right)=k^k\in\mathbb{F}_p.$$
	\end{lemma}
	
	\begin{proof}
		It is clear that 
		\begin{align}
			\prod_{1\le i<j\le k}\left(a_j-a_i\right)\left(\frac{1}{a_j}-\frac{1}{a_i}\right)
			&=\prod_{1\le i<j\le k}\left(a_j-a_i\right)\left(a_i-a_j\right)/(a_ia_j)\nonumber\\
			&=\prod_{1\le i\neq j\le k}\left(a_j-a_i\right)\times \prod_{1\le i<j\le k}\frac{1}{a_ia_j}\label{Eq. a in the proof of Lem. product involving ai}.
		\end{align}
	
	Let $S_1=\prod_{1\le i\neq j\le k}\left(a_j-a_i\right)$ and let $S_2=\prod_{1\le i<j\le k}\frac{1}{a_ia_j}$. We first consider $S_1$. Let 
	$$G_k(t)=\prod_{i=1}^k(t-a_i)\in\mathbb{F}_q[t],$$
	and let $G'_k(t)$ be the formal derivate of $G_k(t)$. Then by the definition of $U_k$ we see that $G_k(t)=t^k-1$. By this we have $G_k'(t)=kt^{k-1}$ and $\prod_{1\le j\le k}a_j=(-1)^{k+1}$. 
	
	In view of the above, one can verify that 
	\begin{align}
		S_1
		&=\prod_{1\le i\neq j\le k}\left(a_j-a_i\right)=\prod_{1\le j\le k}G_k'(a_j)\nonumber\\
		&=\prod_{1\le j\le k}ka_j^{k-1}\nonumber\\
		&=k^k(-1)^{k+1}.\label{Eq. b in the proof of Lem. product involving ai}
	\end{align}
	
	We turn to $S_2$. It is clear that 
	\begin{equation}\label{Eq. c in the proof of Lem. product involving ai}
		S_2=\prod_{1\le i<j\le k}\frac{1}{a_ia_j}=\prod_{1\le j\le k}\frac{1}{a_j^{k-1}}=(-1)^{k+1}.
	\end{equation}
	Combining Eq. (\ref{Eq. a in the proof of Lem. product involving ai}) with Eq. (\ref{Eq. b in the proof of Lem. product involving ai}) and Eq. (\ref{Eq. c in the proof of Lem. product involving ai}), we obtain 
	$$
	\prod_{1\le i<j\le k}\left(a_j-a_i\right)\left(\frac{1}{a_j}-\frac{1}{a_i}\right)=S_1S_2=k^k\in\mathbb{F}_p.
	$$
	
	This completes the proof.	
	\end{proof}
	
	{\noindent\bf Proof of Theorem \ref{Thm. the determinant of Dk}.} As $\char(\mathbb{F}_q)=p>2$, the subset $\{1,-1\}\subseteq\mathbb{Z}$ can be naturally viewed as a subset of $\mathbb{F}_q$. Now one can verify that 
	\begin{align}
    \det D_k
    &=\det\left[(a_i+a_j)^{\frac{q-3}{2}}\right]_{1\le i,j\le k}\nonumber\\
    &=\prod_{i=1}^ka_i^{\frac{q-3}{2}}\det\left[\left(1+\frac{a_j}{a_i}\right)^{\frac{q-3}{2}}\right]_{1\le i,j\le k}\nonumber\\
    &=(-1)^{(k+1)(q-3)/2}\det\left[\left(1+\frac{a_j}{a_i}\right)^{\frac{q-3}{2}}\right]_{1\le i,j\le k}\label{Eq. A in the proof of Thm.1.1}.
	\end{align}
	
	The last equality follows from $\prod_{1\le j\le k}a_j=(-1)^{k+1}$. Let 
	$$f_k(t)
	=\sum_{s=0}^{k-1}
	\left(\sum_{r=0}^{\lfloor\frac{q-3-2s}{2k}\rfloor}	\binom{(q-3)/2}{s+rk}\right)t^s\in\mathbb{F}_p[t]$$
	with $\deg(f_k)\le k-1$. Noting that $(a_j/a_i)^{k+s}=(a_j/a_i)^s$ for any integer $s$, by Eq. (\ref{Eq. A in the proof of Thm.1.1}) we have 
	\begin{equation}\label{Eq. B in the proof of Thm.1.1}
    \det D_k
    =(-1)^{(k+1)(q-3)/2}\times\det\left[f_k\left(\frac{a_j}{a_i}\right)\right]_{1\le i,j\le k}.		
	\end{equation}
	Let
	$$w_k:=\prod_{s=0}^{k-1}\sum_{r=0}^{\lfloor\frac{q-3-2s}{2k}\rfloor}	\binom{(q-3)/2}{s+rk}\in\mathbb{F}_p.$$
	Then by Lemma \ref{Lem. K} and Lemma \ref{Lem. product involving ai} we obtain 
	$$\det D_k=(-1)^{(k+1)(q-3)/2}\cdot w_k\cdot \prod_{1\le i<j\le k}\left(a_j-a_i\right)\left(\frac{1}{a_j}-\frac{1}{a_i}\right)=(-1)^{(k+1)(q-3)/2}\cdot w_k\cdot k^k\in\mathbb{F}_p.$$
	This completes the proof. \qed
	
	We next prove Corollary \ref{Cor. B of Thm.1.1}. 
	
	{\noindent\bf Proof of Corollary \ref{Cor. B of Thm.1.1}.} By Theorem \ref{Thm. the determinant of Dk} if $k=(q-1)/2$, then we have 
	\begin{align}
		\det D_{(q-1)/2}
		&=(-1)^{(q-3)/2}\cdot\prod_{s=0}^{(q-3)/2}\binom{(q-3)/2}{s}\cdot (-1)^{(q-1)/2}\left(\frac{1}{2}\right)^{(q-1)/2}\nonumber\\
		&=-1\cdot\prod_{s=0}^{(q-3)/2}\binom{(q-3)/2}{s}\cdot\phi(2).\label{Eq. A in the proof of Cor.B}
	\end{align}
	The last equality follows from 
	$$\left(\frac{1}{2}\right)^{(q-1)/2}=\phi\left(\frac{1}{2}\right)=\phi(2).$$
	
	We now divide the remaining proof into the following cases.
	
	{\bf Case 1.} $q\equiv1\pmod4$. 
	
	In this case we have $\sqrt{-1}\in\mathbb{F}_q$, where $\sqrt{-1}$ is an element in the algebraic closure of $\mathbb{F}_q$ such that $(\sqrt{-1})^2=-1$. 
	As
	$$2=-\sqrt{-1}\left(1+\sqrt{-1}\right)^2,$$
	we have $\phi(2)=\phi(-\sqrt{-1})$ and hence 
	\begin{equation}\label{Eq. B in the proof of Cor.B}
		\phi(2)=\phi(-\sqrt{-1})=\left(-\sqrt{-1}\right)^{(q-1)/2}=(-1)^{(q-1)/4}.
	\end{equation}
	Combining Eq. (\ref{Eq. A in the proof of Cor.B}) with Eq. (\ref{Eq. B in the proof of Cor.B}) and noting that 
	$$\binom{(q-3)/2}{s}=\binom{(q-3)/2}{(q-3)/2-s},$$ 
	we obtain 
	\begin{equation}\label{Eq. C in the proof of Cor.B}
	\det D_{(q-1)/2}=(-1)^{(q+3)/4}\prod_{s=0}^{(q-5)/4}\binom{(q-3)/2}{s}^2.
	\end{equation}
	This proves the case $q\equiv1\pmod4$. 
	
	{\bf Case 2.} $q\equiv3\pmod4$ and $q>3$. 
	
	In this case since $q\equiv3\pmod4$ we have 
	$$(1+\sqrt{-1})^q=1+(\sqrt{-1})^q=1-\sqrt{-1}.$$
	This, together with $2=-\sqrt{-1}\left(1+\sqrt{-1}\right)^2$, implies that 
	\begin{align}
		\phi(2)
		&=2^{(q-1)/2}=\left(-\sqrt{-1}\right)^{(q-3)/2}(-\sqrt{-1})\left(1+\sqrt{-1}\right)^{q-1}\nonumber\\
		&=(-1)^{(q-3)/4}(-\sqrt{-1})\frac{1-\sqrt{-1}}{1+\sqrt{-1}}\nonumber\\
		&=(-1)^{(q+1)/4}.\label{Eq. D in the proof of Cor.B}
	\end{align}
	Combining Eq. (\ref{Eq. A in the proof of Cor.B}) with Eq. (\ref{Eq. D in the proof of Cor.B}) we obtain that 
	\begin{equation}\label{Eq. E in the proof of Cor.B}
		\det D_{(q-1)/2}=(-1)^{(q+5)/4}\binom{(q-3)/2}{(q-3)/4}\prod_{s=0}^{(q-7)/4}\binom{(q-3)/2}{s}^2.
	\end{equation}
	This proves the case $q\equiv3\pmod4$ and $q>3$. 
	
	Now we assume that $q=p$ is an odd prime. Suppose first $p\equiv1\pmod4$. 
	Then by Eq. (\ref{Eq. C in the proof of Cor.B}) we see that $\det D_{(q-1)/2}$ is a non-zero square in $\mathbb{F}_p$, i.e., $(\frac{\det D_{(p-1)/2}}{p})=1$. 
	In the case $p\equiv3\pmod4$ and $p>3$, by Eq. (\ref{Eq. E in the proof of Cor.B}) and $(\frac{-2}{p})=(\frac{-1/2}{p})=(-1)^{(p+5)/4}$, we have 
	\begin{align*}
	\left(\frac{\det D_{(q-1)/2}}{p}\right)&=(-1)^{(p+5)/4}\left(\frac{\frac{p-3}{2}!}{p}\right)\\
	                      &=(-1)^{(p+5)/4}\left(\frac{\frac{p-1}{2}!}{p}\right)\left(\frac{\frac{-1}{2}}{p}\right)\\
	                      &=\left(\frac{\frac{p-1}{2}!}{p}\right)\\
	                      &=(-1)^{(h(-p)+1)/2}.
	\end{align*}
    The last equality follows from Mordell's result \cite{Mordell} which states that 
    $$\frac{p-1}{2}!\equiv (-1)^{\frac{h(-p)+1}{2}}\pmod p$$
    whenever $p\equiv3\pmod4$ and $p>3$.
    
    In view of the above, we have completed the proof.\qed
	
	\section{Proof of Theorem \ref{Thm. the legendre symbol of Dk, k odd}}
	
	To prove Theorem \ref{Thm. the legendre symbol of Dk, k odd}, we first need the following known result. 
	
    \begin{lemma}\label{Lem. sum of rth power of x over Fq}
    	Let $\mathbb{F}_q$ be the finite field of $q$ elements and let $r$ be a positive integer. Then 
    	\begin{equation*}
    		\sum_{x\in\mathbb{F}_q}x^r=\begin{cases}
    			 0  & \mbox{if}\ q-1\nmid r,\\
    			-1  & \mbox{if}\ q-1\mid r.
    		\end{cases}
    	\end{equation*}
    \end{lemma}
	
	We will see later in the proof that $\det D_k$ has close relations with the determinant of a circulant matrix. Hence we now introduce the definition of circulant matrices. Let $R$ be a commutative ring. Let $b_0,b_1,\ldots,b_{s-1}\in R$. We define the circulant matrix $C(b_0,\ldots,b_{s-1})$ to be an $s\times s$ matrix whose ($i,j$)-entry is $b_{j-i}$ where the indices are cyclic module $s$, i.e., $b_i=b_j$ whenever $i\equiv j\pmod s$. The third author \cite[Lemma 3.4]{W21}  obtained the following result.
	
	\begin{lemma} \label{lemma Wu ffa}
		Let $R$ be a commutative ring. Let $s$ be a positive integer. Let $b_0,b_1,\ldots,b_{s-1}\in R$ such that
		\begin{align*}
			b_i=b_{s-i} \ \ \t{for each $1\ls i\ls s-1$.}
		\end{align*}
		If $s$ is even, then there exists an element $u\in R$ such that 
		$$
		\det C(b_0,\ldots,b_{s-1})=\left(\sum_{i=0}^{s-1}b_i\right)\left(\sum_{i=0}^{s-1}(-1)^ib_i\right)\cdot u^2.
		$$
		If $s$ is odd, then there exists an element $v\in R$ such that
		$$
		\det C(b_0,\ldots,b_{s-1})=\left(\sum_{i=0}^{s-1}b_i\right)\cdot v^2.
		$$
	\end{lemma}
	
	{\noindent\bf Proof of Theorem \ref{Thm. the legendre symbol of Dk, k odd}.} 
	As $k$ is odd, we have $2\mid (q-1)/k$. For simplicity, we let $q-1=nk=2mk$. Since $k\mid (q-1)/2$ in this case, $\phi(a_i)=a_i^{(q-1)/2}=1$ for each $a_i\in U_k$.  Let $g$ be a generator of the cyclic group $\mathbb{F}_q^{\times}$. By the above one can verify that 
	\begin{align}
		\det D_k
		&=\prod_{i=1}^ka_i^{(q-3)/2}\det\left[\left(1+\frac{a_j}{a_i}\right)^{(q-3)/2}\right]_{1\le i,j\le k}\nonumber\\
		&=\det\left[\left(1+g^{nj-ni}\right)^{(q-3)/2}\right]_{0\le i,j\le k-1}\nonumber
	\end{align}
	The last equality follows from 
	\begin{equation*}
		\prod_{i=1}^ka_i=(-1)^{k+1}=1.
	\end{equation*}
	By the above and using the properties of determinants, one can verify that 
	\begin{equation}\label{Eq. A in the proof of Thm.1.2}
		\det D_k=\det\left[\left(1+g^{nj-ni}\right)^{(q-3)/2}g^{mj-mi}(-1)^{j-i}\right]_{0\le i,j\le k-1}.
	\end{equation}

    For each $0\le i\le k-1$ we set 
    $$b_i=\left(1+g^{ni}\right)^{(q-3)/2}g^{mi}(-1)^i.$$
	We claim that $b_i=b_{k-i}$ for any $1\le i\le k-1$. In fact, for any $1\le i\le k-1$, noting that 
	$$g^{km}=\phi(g)=-1,\ g^{nk}=1,\ 2\nmid k,\ \text{and}\ \left(\frac{1}{g^{ni}}\right)^{(q-3)/2}=g^{ni},$$
	one can verify that 
	\begin{align*}
		b_{k-i}&=\left(1+g^{nk-ni}\right)^{(q-3)/2}g^{mk-mi}(-1)^{k-i}\\
		       &=\left(\frac{1+g^{ni}}{g^{ni}}\right)^{(q-3)/2}g^{-mi}(-1)^i\\
		       &=\left(1+g^{ni}\right)^{(q-3)/2}g^{(n-m)i}(-1)^i\\
		       &=b_i.
	\end{align*}
Hence by Eq. (\ref{Eq. A in the proof of Thm.1.2}) we have 
$$\det D_k=\det C(b_0,b_1,\cdots,b_{k-1}).$$
	Now by Lemma \ref{lemma Wu ffa} and Eq. (\ref{Eq. A in the proof of Thm.1.2}) we obtain that 
	\begin{equation}\label{Eq. D in the proof of Thm1.2}
		\det D_k=\left(\sum_{i=0}^{k-1}b_i\right)v^2
	\end{equation}
	for some $v\in\mathbb{F}_q$. Now we consider the sum $\sum_{i=0}^{k-1}b_i$. 
	It is easy to verify that 
	\begin{align}
	\sum_{i=0}^{k-1}b_i
	&=\sum_{i=0}^{k-1}\left(1+g^{ni}\right)^{(q-3)/2}g^{mi}(-1)^i\nonumber\\
	&=\sum_{i=0}^{k-1}\left(1+g^{ni}\right)^{(q-3)/2}g^{mi}g^{mki}\nonumber\\
	&=\frac{1}{n}\sum_{x\in\mathbb{F}_q}\left(1+x^n\right)^{(q-3)/2}x^{m+mk}\nonumber\\
	&=\frac{1}{n}\sum_{r=0}^{mk-1}\binom{(q-3)/2}{r}\sum_{x\in\mathbb{F}_q}x^{m+mk+2mr}.
	\label{Eq. E in the proof of Thm1.2}
	\end{align}
	Now by Lemma \ref{Lem. sum of rth power of x over Fq} and $2\nmid k$,  we see that 
	$$\sum_{x\in\mathbb{F}_q}x^{m+mk+2mr}=\begin{cases}
		 0 & \mbox{if}\ k\nmid 1+2r,\\
		-1 & \mbox{if}\ k\mid 1+2r.
	\end{cases} $$
   Applying this and Lemma \ref{Lem. sum of rth power of x over Fq} to Eq. (\ref{Eq. E in the proof of Thm1.2}) and noting that $-1/n=k$ in $\mathbb{F}_p$, we obtain 
   \begin{equation}\label{Eq. F in the proof of Thm1.2}
   	s_k:=\sum_{i=0}^{k-1}b_i=k\sum_{r=1}^m\binom{(q-3)/2}{((2r-1)k-1)/2}.
   \end{equation}
	Suppose that $D_k$ is non-singular. Then by Theorem \ref{Thm. the determinant of Dk} we have $\det D_k\in\mathbb{F}_p^{\times}$. Hence by Eq. (\ref{Eq. D in the proof of Thm1.2}) and Eq. (\ref{Eq. F in the proof of Thm1.2}) we have 
	$$\left(\frac{\det D_k}{p}\right)=\left(\frac{s_k}{p}\right).$$
	This completes the proof. \qed
	
    \section{Proof of Theorem \ref{Thm. the determinant of Tk}}
	
	{\noindent\bf Proof of Theorem \ref{Thm. the determinant of Tk}.} It is clear that 
	\begin{align}
		\det T_k
		&=\prod_{i=1}^{k}(a_i^2)^{(q-3)/2}\times
		\det\left[\left(1+\frac{a_j}{a_i}+\left(\frac{a_j}{a_i}\right)^2\right)^{(q-3)/2}\right]_{1\le i,j\le k}\nonumber\\
		&=\det\left[\left(1+\frac{a_j}{a_i}+\left(\frac{a_j}{a_i}\right)^2\right)^{(q-3)/2}\right]_{1\le i,j\le k}.\label{Eq. A in the proof of Thm.1.3}
	\end{align}
	The last equality follows from 
	$$\prod_{i=1}^ka_i=(-1)^{k+1}.$$
	Let 
	$$g_k(t)=
	\sum_{s=0}^{k-1}
	\left(\sum_{r=0}^{\lfloor\frac{q-3-s}{k}\rfloor}
	\binom{(q-3)/2}{s+rk-(q-3)/2}_2\right)t^s\in\mathbb{F}_p[t]$$
	with $\deg(g_k)\le k-1$. Then by Eq. (\ref{Eq. A in the proof of Thm.1.3}), Lemma \ref{Lem. K} and the definition of trinomial coefficients we have 
	\begin{align*}
		\det T_k
		&=\det\left[g_k\left(\frac{a_j}{a_i}\right)\right]_{1\le i,j\le k}\\
		&=\prod_{1\le i<j\le k}\left(a_j-a_i\right)\left(\frac{1}{a_j}-\frac{1}{a_i}\right)\times\prod_{s=0}^{k-1}\sum_{r=0}^{\lfloor\frac{q-3-s}{k}\rfloor}
		\binom{(q-3)/2}{s+rk-(q-3)/2}_2\\
		&=l_kk^k\in\mathbb{F}_p.		
	\end{align*}
	The last equality follows from Lemma \ref{Lem. product involving ai}. 
	
	In view of the above, we have completed the proof.\qed
	
	\Ack\ We would like to thank the referee for helpful comments. This work was supported by the Natural Science Foundation of China (Grant No. 12101321).

\end{document}